\documentclass[11pt]{article}

\usepackage{pdfsync}

\usepackage{a4wide}
\usepackage{amsfonts,amssymb}
\usepackage{amsmath}
\usepackage{enumerate}
\usepackage{amsthm}

\usepackage[normalem]{ulem}

\newcommand{\R}{\mathbb{R}}
\renewcommand{\L}{\mathbb{L}}

\newcommand{\ep}{\varepsilon}

\newcommand{\dt}{\partial_t}

\newcommand{\grad}{\nabla} 
\newcommand{\produ}[1]{\left\langle #1\right\rangle}

\newcommand{\nab}{\bar\nabla}
\newcommand{\te}{\widetilde{\ep}}
\newcommand{\W}{\mathbb{H}^n}

\newtheorem{lemma}{Lemma}
\newtheorem{proposition}{Proposition}
\newtheorem{theorem}{Theorem}
\newtheorem{corollary}{Corollary}
\newtheorem{definition}{Definition}
\newtheorem{examp}{Example}
\newtheorem{remark}{Remark}

\title{Translating solitons from semi-Riemannian foliations}
\author{
\begin{tabular}{ccc}
Marie-Am\'elie Lawn &  & Miguel Ortega\\
\small Imperial College (UK) & & 
\small Universidad de Granada (Spain) \\
\small   m.lawn@imperial.ac.uk & &\small miortega@ugr.es
\end{tabular}
}
\date{\today}

\begin{document}
\maketitle

\begin{center} \textit{To the memory of Antonio.} \end{center}

\begin{abstract}
We recall the notion of (vertical) translating solitons in a product of a semi-Riemannian manifold $(M,g)$ and the real line. Mainly, we restrict our attention to those which are the graph of a smooth function. When dealing with submersions, we show a criteria to lift (or project) translating solitons from the base manifold to the total space (or viceversa). In particular, manifolds foliated by codimension 1 orbits of a Lie group action give rise to such solitons, up to solving a first-order ordinary differential equation. This gives us explicit criteria under which the graph of a function is a soliton, and we employ them to construct many examples of solitons, both new and old, in a unified way.
\end{abstract}

\noindent \textit{Keywords:} Translating soliton, submersions, Riemannian manifolds, semi-Riemannian manifolds, Lie group.

\noindent \textit{MSC[2010] Classification:} 53C44, 53C21, 53C42, 53C50.

\section{Introduction}

Given a smooth manifold $M$, assume a family of smooth immersions in a semi-Rieman\-nian manifold $(\mathbf{M},\mathbf{g})$, $F_t:M\rightarrow \mathbf{M}$, $t\in [0,\delta)$, $\delta>0$, with mean curvature vector $\vec{H}_t$. The initial immersion $F_0$ is called a solution to the \textit{mean curvature flow} (up to local diffeomorphism) if 
\begin{equation}\label{MCF} \left(\frac{d}{dt} F_t\right)^{\perp} = \vec{H}_t,
\end{equation}
where $\perp$ means the orthogonal projection on the normal bundle. In the Euclidean and Minkowski space, there is a famous family of such immersions, namely, translating solitons. A submanifold is called \textit{translating soliton} in the Euclidean Space when its mean curvature $\vec{H}$ satisfies the following equation:
\begin{equation} \label{soliton}
\vec{H} = v^{\perp},
\end{equation}
for some constant unit vector $v\in\R^{n+1}$. Indeed, if a submanifold $F:M\rightarrow \mathbb{R}^{n+1}$ satisfies this condition, then it is possible to define the forever flow $\Gamma:M\times[0,+\infty)\rightarrow \mathbb{R}^{n+1}$, $\Gamma(p,t)=F_t(p)=F(p)+tv$. Clearly, 
\[ \left(\frac{d}{dt} F_t\right)^{\perp} = v^{\perp}=\vec{H}.
\]
This justifies our definition. Such solutions have been widely studied in the case where the ambient space is the Euclidean (or the Minkowski) space.  Probably, the most famous examples are the Grim Reaper curve in $\R^2$, and the translating paraboloid and translating catenoid, \cite{CSS}. For a good list of other examples, see \cite{MSHS}. Also, in \cite{N} there are some examples with complicated topology. In  \cite{ALR}, the authors focus on the weak maximum principle applied to these objetcs. 

The starting point of this paper is the fact that the translating paraboloid and translating catenoid are rotationally symmetric, namely, invariant by the Lie group $SO(n)$ acting by isometries. We revise the concept of translating solitons when the ambient manifold is a semi-Riemannian product of a semi-Riemannian manifold and the real line. Needless to say, this includes the Riemannian setting.  Second, we are interested in  constructing  translating solitons having a nice behaviour under the action of Lie groups. We show a technique to construct such examples in different semi-Riemannian manifolds.  More precisely, assume that $(M,g)$ is a connected semi-Riemann manifold of dimension $n\geq 2$ and index $0\leq \alpha  \leq n-1$.  Given $\ep=\pm 1$, we construct the semi-Riemannian product $\bar{M}=M\times\R$ with metric $\produ{,}=g+\ep\mathrm{d}t^2$. The vector field $\dt\in\mathfrak{X}(\bar{M})$ is obviously Killing and unit, spacelike when $\ep=+1$ and timelike when $\ep=-1$. Now, let $F:\Gamma\rightarrow \bar{M}$ be a submanifold with mean curvature vector $\vec{H}$. Denote by $\dt^{\perp}$ the normal component of $\dt$ along $F$.  
\begin{definition} With the previous notation, we will call $F$  a {\normalfont (vertical) translating soliton of the mean curvature flow}, or simply, a {\normalfont translating soliton}, if $\vec{H}=\dt^{\perp}$.
\end{definition}
To justify this definition we remark that if $F(x)=(F_1(x),F_2(x))$ for any $x\in \Gamma$, we can define the map 
\[ 
F:\Gamma \times \R \rightarrow \bar{M}, \quad F(x,t)=(F_1(x),F_2(x)+t).
\]
Hence each $F(-,t)$ is a submanifold with associated mean curvature vector $\vec{H}_t$ and therefore, up to tangential diffeomorphisms, provides a forever solution to \eqref{MCF} invariant by the Killing vector field $\dt$.  This clearly generalizes the classical definition of translating solitons in $\R^{n+1}$. Notice that due to the rich group of isometries of $\R^{n+1}$, considering any constant unit vector field $v$ such that $\vec{H}=v^{\perp}$ is equivalent to considering $v=\partial_t$.

In this paper, we will focus on \textit{graphical} translating solitons. Namely, given $u\in C^2(M)$, we construct its graph map $F:M\rightarrow M\times\R$, $F(x)=(x,u(x))$. If $\nu$ is the upward normal vector along $F$ with $\ep'=\mathrm{sign}(\produ{\nu,\nu})=\pm 1$, we characterize in Proposition \ref{basicequation} the graphical translating solitons in $(\bar{M},\produ{,})$ as those satisfying the following PDE:
\begin{equation}
\label{fundamentalequation} \mathrm{div}\left(\frac{ \grad u } { \sqrt{\ep'\big(\ep+\vert \grad u\vert_g^2\big)}} \right)=
\frac{1}{ \sqrt{\ep'\big(\ep+\vert \grad u\vert_g^2\big)}}. 
\end{equation}
This equation is in general difficult to solve. The geometric way to reduce the number of variables is to consider on the manifold $M$ a foliation by orbits of a Lie group action. But we can often understand the actions by Lie groups as particular cases of submersions. Thus, we first consider the general situation where $M$ submerses to a base manifold $B$ and prove the following result:
\begin{center}
\parbox{0.88\textwidth}{
\textit{
Let $\pi:(M,g_M)\rightarrow (B,g_B)$ be a harmonic semi-Riemannian submersion. Given $u \in C^2(B)$, let $F:B\rightarrow B\times\R$, $F(x)=(x,u(x))$ be its graph map, and $\widetilde{F}:M\rightarrow M\times\R$ be the map $\widetilde{F}(x)=(x,\pi\circ u(x))$. Then $\widetilde{F}$ is a graphical translating soliton if and only if $F$ is a graphical  translating soliton.
}}
\end{center}
In fact, we obtain a more general result in Theorem \ref{bundle2}. Indeed, if we take into account the mean curvature of the fibers in the direction of the normal to the map $\tilde{F}$, we show that the existence of the translating soliton $\widetilde{F}$ in $M$ is equivalent to the existence of what we call an \textit{H-perturbed translating soliton $F$} in $B$. 

Moreover, an important particular case arise when $\pi:M\rightarrow B=I$, where  $B=I$ is an open interval with metric $\te ds^2$ ($\te =\pm 1$). Clearly,  the fibers of the submersion will be hypersurfaces, and we assume that each fiber $\pi^{-1}(s)$, $s\in I$, has constant mean curvature $h(s)$. Since any function $u$ defined on $M$ which is constant along the fibers, projects to another function $f:I\rightarrow\R$ such that $u=f\circ \pi$, we show in Theorem \ref{equiv_ODE} that equation \eqref{fundamentalequation} reduces to the following ODE, 
\begin{equation} \label{keyODE}
f''(s) = \big(\te +\ep f'(s)^2\big) \big( 1 - f'(s)h(s) \big). 
\end{equation}

Next,  we let a Lie group $\Sigma$ act on the manifold in a nice way. In \cite{Mos57}, it is shown that when $\Sigma$ is compact, and there is at least a good orbit with codimension 1, then, the space of orbits is either $\mathbb{S}^1$ or an interval. We will work in the general setting where $\Sigma$ acts by isometries and the projection to the space of orbits $\pi:M\rightarrow M/\Sigma$ is a well-defined smooth map, where $M/\Sigma$ is diffeomorphic to an open interval. We show that if the gradient of the projection is nowhere lightlike, then up to composing with another function, $\pi$ can be assumed to be a semi-Riemannian submersion with constant mean curvature fibers. We state here a more geometric and milder version of Theorem \ref{existence}, which is a consequence of Theorem \ref{equiv_ODE}.
\begin{center}
\parbox{0.88\textwidth}{
\textit{Let $(M,g)$ be a connected semi-Riemannian manifold. Let $\Sigma$ be a Lie group acting by isometries on $M$ and $\pi:(M,g_M)\rightarrow (I,\widetilde{\varepsilon}ds^2)$ be a semi-Riemannian submersion such that the fibers of $\pi$ are orbits of the action, with function $h:I\to\R$ representing the mean curvature of the fibers. Given $u\in C^2(M)$, consider its graph map 
$F:M\rightarrow M\times\R$, $F(x)=(x,u(x))$ for any $x\in M$. Then, $F$ is a $\Sigma$-invariant translating soliton if, and only if, there exists a solution $f\in C^2(I,\R)$ to \eqref{keyODE} such that $u=f\circ\pi$. 
}}
\end{center}
This leads to the existence of translating solitons which are the union of two graphical ones, as stated in Corollary \ref{oide}. We recall that the \textit{translating catenoid} in $\R^3$  is obtained in this way in \cite{CSS}, and it is one of the main classical examples. Thus, we will be able to obtain translating solitons with two ends in some manifolds. Second, we show in Corollary \ref{globalexamples} that, under certain assumptions on the signature and for certain values of the initial value problems, solutions exist on the whole manifold $M$. 

Last Section is devoted to obtaining specific examples. All of them share a common technique: We start with a manifold $M$ and a Lie group of isometries $\Sigma$ such that the space of orbits $M/\Sigma$ is an open interval. The associated ODE \eqref{keyODE} deeply relies on the auxiliary function $h:M/\Sigma\rightarrow\R$, which depends on the case. When the ODE becomes a problem with a singularity, we show the existence of local solutions 
by using dynamical systems, and we lift them up to the original manifold $M$ to obtain translating solitons. 
Inspired by the Grim Reaper Cylinder in $\R^n$, we construct many others by using product manifolds. Needless to say, we recover the already known rotationally symmetric translating solitons in the Euclidean and Minkowski Spaces, but we obtain among others a new \textit{translating catenoid} in Minkowski Space,   rotationally symmetric translating solitons in the De Sitter Space and the Hyperbolic Space. By a gluing technique from \cite{BCO}, we construct a $C^{\infty}$, boost invariant, translating soliton in Minkowski Space by gluing (up to) 4 pieces. The goal is not to provide an exhaustive list, but rather to illustrate our technique, and make it useful for future works.

\section{Setup}
The following proposition is well-known in the Euclidean setting (see for example \cite{LTW}). The proof is similar in our context, but we include it for completeness and discuss the necessary modifications.
\begin{proposition}\label{basicequation}
Let $(M,g)$ be a semi-Riemannian manifold, $u:M\rightarrow \R$ a\, $\mathcal{C}^2$ function, and let $F:M\rightarrow M\times\R=:\bar{M}$, $F(x)=(x,u(x))$ be  its graph map. Given $\ep=\pm 1$, assume that $F:(M,\gamma=F^*\produ{,})\rightarrow (\bar{M},\produ{,}=g+\ep dt^2)$ is a semi-Riemannian hypersurface with unit upward normal $\nu$ such that $\produ{\nu,\nu}=\ep'=\pm 1$. Then, $F$ is a (vertical) translating soliton  if, and only if, function $u$ satisfies \eqref{fundamentalequation}.
\end{proposition}
\begin{proof}
Note that, under the usual identifications, for each $X\in TM$, we have
\[ dF(X) = (X,du(X))  = (X,0)+du(X)\dt = (X,g(\grad  u,X)) = 
(X,0) + g(\grad  u,X)\dt,
\]
where $\grad  u$ is the $g$-gradient of $u$.  We consider the metric $\gamma=F^*\produ{,}$ on $M$, but we have to assume that $(M,\gamma)$ is a semi-Riemannian submanifold. In particular, the following function is constant, $\ep'=\mathrm{sign}\big(\ep+\vert\grad u\vert_g^2\big)=\pm 1$. 
Therefore,  the upward normal vector field is
\begin{equation}\label{normal} \nu = \frac{1}{W}\big(-\ep(\grad u,0)+\dt\big), \quad W=+\sqrt{\ep'\big(\ep+\vert\grad u\vert_g^2\big)}.
\end{equation}
Needless to say, $\produ{\nu,\nu}=\ep'$. Then, we consider a local $g$-orthonormal frame $B=(e_1,\ldots,e_n)$ such that $g(e_i,e_j)=\ep_i\delta_{ij}$, for any $i,j=1,\ldots,n$. We denote $u_i=du(e_i)$, $i=1,\ldots,n$. For this frame, we compute the induced metric $\gamma=F^{*}\produ{,}$, so that the coefficients of the Gram matrix are 
\[ \gamma_{ij}=\gamma(e_i,e_j) = \ep_i\delta_{ij}+\ep u_iu_j. 
\]
Then, the inverse matrix is     
\[ \gamma^{ij} = \ep_i\delta_{ij} -\frac{\ep'}{W^2}\ep_i\ep_ju_iu_j.
\]
Now, if $II$ is the second fundamental form of $F$, then the mean curvature vector of $F$ is 
\[ \vec{H} = \mathrm{tr}_{\gamma}(II) = \ep' \produ{\vec{H},\nu} \nu 
=\ep' \produ{\dt^{\perp},\nu}\nu =  \ep' \produ{\dt,\nu}\nu 
=\frac{\ep\ep'}{W}\nu. 
\] 
But in out setting, what really matters is $ \produ{\nu,\vec{H}} =\frac{\ep}{W}.$ On the other hand, let  $\nab$ the Levi-Civita connection of $(\bar{M},\produ{,})$. We recall O'Neill's book \cite{ON}, and its equations for the Levi-Civita connection of a (warped) product. Thus,
\begin{align*}
&\produ{\nu,\vec{H}}  =\produ{\nu,\mathrm{tr}_\gamma(II)} 
=\sum_{i,j}\gamma^{ij}\produ{\nu,\nab_{dF(e_i)}dF(e_j)} 
 = \sum_{i,j}\gamma^{ij} \frac{\ep}{W} du_j(e_i) \\
& = \frac{\ep}{W}\mathrm{div}(\grad u) -\ep\ep'\sum_{i,j}\frac{\ep_i\ep_ju_iu_j}{W^3}e_i(e_j(u)).
\end{align*}
Now, we compute
\[ \produ{\grad\Big(\frac{1}{W}\Big),\grad u } 
=\frac{-1}{W^2}\sum_i \ep_iu_ie_i\Big( \sqrt{\ep'\big(\ep+\vert\grad u\vert_g^2\big)} \Big) 
= \frac{-\ep'}{W^3} \sum_{i,j} \ep_i\ep_ju_iu_je_i(e_j(u)).
\]
All together, 
\[\produ{\nu,\vec{H}}  =\frac{\ep}{W}\mathrm{div}(\grad u) +\ep \produ{\grad\Big(\frac{1}{W}\Big),\grad u } =
\ep \mathrm{div}\Big( \frac{\grad u}{W}\Big).
\]
\end{proof}
\begin{remark}\normalfont 
If we call $H=\mathrm{tr}(A)/n$ the mean curvature function of the hypersurface, $A$ being the shape operator, then $\produ{\nu,\vec{H}}=n H \ep'$. In particular, 
\[
\mathrm{div}\left(\frac{ \grad u } { \sqrt{\ep'\big(\ep+\vert \grad u\vert_g^2\big)}} \right)=
\frac{1}{ \sqrt{\ep'\big(\ep+\vert \grad u\vert_g^2\big)}} = \ep' n H. 
\]
\end{remark}
\begin{corollary} Let $(M,g)$ be a compact, without boundary, orientable Riemannian manifold. Then, $M$ does not admit any globally defined graphical Translating Soliton $F:M\rightarrow (M\times\mathbb{R},g+\varepsilon dt^2)$, for $\varepsilon=\pm 1$.
\end{corollary}
\begin{proof} Assume that there exists a globally defined graphical Translating Soliton on $M$. Then, for some function $u\in C^2(M)$,  equation \eqref{fundamentalequation} holds true. By using the volume form $d\mu_g$, we obtain
\[ 0 = \int_M \mathrm{div}\left(\frac{\nabla u}{\sqrt{\varepsilon'(1+\vert \nabla u\vert^2)}}\right) d\mu_g = 
\int_M \frac{1}{\sqrt{\varepsilon'(1+\vert \nabla u\vert^2)}} d\mu_g >0.
\]
This is a contradiction.
\end{proof}

\section{Submersions}
We first recall shortly some facts about submersions.  For more details we refer to {\cite{Escobales} and the classical O'Neill's book \cite{ON}. If $\pi:(M,g_M)\rightarrow(B,g_B)$ is a semi-Riemannian submersion, we can consider the following decomposition of the tangent bundle $TM$  $$TM=\mathcal{H}\oplus\mathcal{V},$$
where $\mathcal{V}:=\ker{d\pi}$ is the vertical distribution consisting of vectors tangent to the fibers and $\mathcal{H}$ is the horizontal complement with respect to $g_M$. Note that $\mathcal{V}$ is integrable, but it is usually not true for the horizontal distribution. 
Every vector field $X$ in $TM$ can then be uniquely
written as $X=\mathcal{H}X+\mathcal{V}X$ where
$\mathcal{H}X$ (resp. $\mathcal{V}X$) is the horizontal (resp.
vertical) component. For completeness and clearness sake, we add the following 
\begin{definition}
A vector field $X\in\Gamma(TM)$ is
called horizontal, if  for all $x\in M$,
$X_{x}\in\mathcal{H}_{x}$ and vertical,
if  for all $x\in M$,
$X_{x}\in\mathcal{V}_{x}$. It is called projectable, if there exists a vector field $\check{X}\in\Gamma(TB)$ such
 that for all $x\in M$,
 $d\pi(X_{x})=\check{X}_{\pi(x)}$, $X$ and
 $\check{X}$ are called $\pi$-related. It is called basic, if it is projectable and horizontal.
\end{definition}
We add the following useful 
\begin{remark}\label{remark_submersions} \normalfont 
\begin{itemize}
\item[1)] Obviously, from the definition of $\mathcal{V}$, a
vector field $X\in\Gamma(TM)$ is vertical, if and only if
it is $\pi$-related to the null section of $TB$.
\item[2)] For every vector field $\check{X}$ in $TB$ there exists a unique horizontal vector field in $\Gamma(TM)$ which is $\pi$-related to $\check{X}$. It is called the horizontal lift of $\check{X}$. We will denote it by $X^h$. Moreover we will denote by $\mathcal{H}(X)$ the horizontal component of a vector field $X$ of $TM$.
\item[3)] If $X$, $Y$ are basic vector fields on $M$ $\pi$-related to $\check{X}$, $\check{Y}$, then
$g_M(X,Y)=\pi^*g_B(\check{X},\check{Y})$. Moreover $\mathcal{H}\nabla^M_XY$ is the basic vector field corresponding to $\pi^*\nabla^B_{\check{X}}\check{Y}$, where $\nabla^M$ (respectively $\nabla^B$) are the Levi-Civita connection on $M$ (respectively $B$).
\end{itemize}
\end{remark}
Finally let $W$, $Z$ be vector fields in $TM$. We recall the definition of the two O'Neill fundamental tensors.
\begin{eqnarray}
T_WZ&:=&\mathcal{H}\nabla^M_{\mathcal{V}W}(\mathcal{V}Z)+\mathcal{V}\nabla^M_{\mathcal{V}W}(\mathcal{H}Z),\nonumber\\
A_WZ&:=&\mathcal{V}\nabla^M_{\mathcal{H}W}(\mathcal{H}Z)+\mathcal{H}\nabla^M_{\mathcal{H}W}(\mathcal{V}Z)
\end{eqnarray}
Note that if $W$ and $Z$ are vertical, $T$ is the second fundamental form of the fibers. Hence  the fibers of the submersion are totally geodesic if and only if $T\equiv 0$. We recall that the submersion is harmonic, if and only if the mean curvature of the fiber $h=0$. Further, let $X$, $Y$ be horizontal vectors, and $U$, $V$ be vertical vectors we recall the following useful formulas
\begin{eqnarray}\label{connection_ONtensors}
\nabla^M_UV&=&T_UV+\hat{\nabla}_UV, \quad \nabla^M_UX=\mathcal{H}\nabla^M_VX+T_UX\\
\nabla^M_XU&=&A_XU+\mathcal{V}\nabla^M_XU,\quad \nabla^M_XY=\mathcal{H}\nabla^M_XY+A_XY
\end{eqnarray}
where $\hat{\nabla}$ is the connection on the fibers.

\begin{theorem}\label{bundle2}
Let $\pi:(M,g_M)\rightarrow (B,g_B)$ be a semi-Riemannian submersion 
with constant mean curvature fibers.
Given $u \in C^2(B)$, let $F:B\rightarrow B\times\R$, $F(x)=(x,u(x))$ be its graph map, and $\widetilde{F}:M\rightarrow M\times\R$ be the map $\widetilde{F}(x)=(x,u\circ \pi(x) )$. Let $\nu_F$, $\nu_{\widetilde{F}}$ be the unit upward vector along $F$ and $\widetilde{F}$, respectively, and $H_{\mathrm{fib}}^{\nu_{\widetilde{F}}}$ be the mean curvature of the fibers of $\pi\times 1:M\times\R\rightarrow B\times\R$ in the direction of $\nu_{\widetilde{F}}$. 
Then $\widetilde{F}$ is a graphical translating soliton if and only if $F$ satisfies the equation $\vec{H}_{F}=\partial_t^{\perp}+H_{\mathrm{fib}}^{\nu_{\widetilde{F}}}\nu_F$.
\end{theorem} 
\begin{proof} Consider the submersion $\pi:(M,g_M)\rightarrow (B,g_B)$.
 $\widetilde{F}(x)=(x,v(x))$ for any $x\in M$, where $v:M\rightarrow \R$ is the height function. Similarly, $F(p)=(p,u(p))$ for $p\in B$, $u:B\rightarrow\R$, and by assumption $u\circ\pi = v$. 

Now locally for every $x_0\in M$, we take an open neighborhood $U$ of $x_0$ and a collection
of projectable vector fields $\{e_i\}_{i=1}^n$ such that $\{e_i(x)\}_{i=1}^n$ is a basis for $T_xM$,
for all $x$ in $U$. Note that we can assume that $e_1,\ldots,e_k$ are vertical
and $e_{k+1},\ldots,e_{n}$ are horizontal w.r.t. $\pi$. Then by Remark \ref{remark_submersions}
\begin{eqnarray*}
g_M((\grad^Bu)^{h},e_i) =\pi^*g_B(\grad^Bu,d\pi(e_i)) = \pi^*du(d\pi(e_i)) = dv(X_i)=g_M(\grad^Mv,e_i),
\end{eqnarray*}
Hence $(\grad^Bu)^h=\grad^Mv$ and, since the submersion is semi-Riemannian, $\vert\grad^M v\vert ^2 = \vert \grad^B u\vert^2$. 

Next, $\pi\times 1:\bar{M}=(M\times \R,g_M+\ep dt^2)\rightarrow \bar{B}=(B\times \R,g_B+\ep dt^2)$ is another semi-Riemannian submersion, satisfying that  the horizontal lift of $\dt$ is $\dt$. Since $\vert\grad^M v\vert ^2 = \vert \grad^B u\vert^2$, we have $W^2=\ep'(\ep+\vert\grad^M v\vert ^2) =\ep'(\ep+ \vert \grad^B u\vert^2)$. 
 Notice that obviously by construction $\nu_{_{\widetilde{F}}}$ is horizontal.
Moreover for any $(p,t)\in B\times \R$, $(\pi\times 1)^{-1}(p,t)=\pi^{-1}(p)\times\{t\}$.
By equation \eqref{normal}, the horizontal lift w.r.t. $\pi\times 1$ of the normal $\nu_{F}$ in $B\times\R$ is exactly the normal $\nu_{\widetilde{F}}$ in $M\times \R$. 

By extending the basis $\{e_i\}_{i=1}^n$ of $TM$, we can now construct the local orthonormal frame \newline $(e_1,\ldots,e_n,e_{n+1}=\dt)$ of $T\bar{M}$, where $\ep_i=g_{\bar{M}}(e_i,e_i)=\pm 1$, depending on the signature of $M$. Note that $\{v_i:=d(\pi\times 1)e_i \vert k+1\leq i\leq n+1\} $ is a local orthonormal frame of $T\bar{B}$ with same signs. Now for the vertical vector fields $e_i$, $1\leq i\leq k$, and the horizontal vector field $\nu_{\widetilde{F}}$ we get using formulas \eqref{connection_ONtensors}
\[ g_{\bar{M}} \big( \nabla_{e_i}^{\bar{M}} \nu_{{\widetilde{F}}},e_i\big) = 
g_{\bar{M}}\big( \mathcal{H}\nabla_{e_i}^{\bar{M}} \nu_{\widetilde{F}},e_i\big)+g_{\bar{M}}\big( T_{e_i}\nu_{\widetilde{F}},e_i\big),
\]
Hence $\nabla_{e_i}^{\bar{M}} \nu_{\widetilde{F}}$ is horizontal if and only if the second term of the right hand side vanishes. Since the vector fields $e_i$ are vertical, they are by definition tangent to the leaves of the submersion, and the horizontal vector field $\nu_{\widetilde{F}}$ is normal to the leaves. Hence $\sum_i^kg_{\bar{M}}\big( T_{e_i}\nu_{\widetilde{F}},e_i\big)$ is exactly the mean curvature $H_{\mathrm{fib}}^{\nu_{\widetilde{F}}}$ of the fibers in the direction of $\nu_{\widetilde{F}}$. 

Denoting hence respectively by $II_{\widetilde{F}}$ and $II_F$ the second fundamental forms of $\widetilde{F}$ and $F$ with associated mean curvature vectors $\vec{H}_{\widetilde{F}}$ and $\vec{H}_F$ we use Remark \ref{remark_submersions} to compute
\begin{align*}
& \vec{H}_{\widetilde{F}} = \sum_i^n \varepsilon_i II_{\widetilde{F}}(e_i,e_i)=
\sum_i ^n\ep_i g_{\bar{M}}(\nabla^{\bar{M}}_{e_i}\nu_{\widetilde{F}},e_i)\nu_{\widetilde{F}}= 
\sum_{i\geq k+1}^n \ep_i g_{\bar{M}}(\nabla^{\bar{M}}_{e_i}\nu_{\widetilde{F}},e_i)\nu_{\widetilde{F}} +H_{\mathrm{fib}}^{\nu_{\widetilde{F}}} \nu_{\widetilde{F}} \\& = 
\sum_{i\geq k+1} \ep_i \pi^*g_{\bar{B}}(\pi^*\nabla^{\bar{B}}_{v_i}\nu_{F},v_i)\nu_{F}^h+H_{\mathrm{fib}}^{\nu_{\widetilde{F}}} \nu_{\widetilde{F}} 
=\Big( \sum_{i\geq k+1} \ep_i g_{\bar{B}}(\nabla^{\bar{B}}_{v_i}\nu_{F},v_i)\nu_{F}\Big)^{h}+H_{\mathrm{fib}}^{\nu_{\widetilde{F}}} \nu_{\widetilde{F}}\\& = \vec{H}_{F}^{h}+H_{\mathrm{fib}}^{\nu_{\widetilde{F}}} \nu_{\widetilde{F}} .
\end{align*}
Projecting and using the translating soliton equation then directly shows our result. Note that the projection of the mean curvature of the fibers to the base only makes sense if it is constant.
\end{proof}

\begin{remark} \normalfont 
We notice that this result especially holds in the case where the submersion has totally geodesic fiber or is harmonic. In fact, the horizontal lift of $\vec{H}_{F}$ is $\vec{H}_{\widetilde{F}}$ if and only if the fibers are minimal with respect to the normal vector $\nu_{\widetilde{F}}$, i.e. $H_{\mathrm{fib}}^{\nu_{\widetilde{F}}}=0$. We point out that only the minimality in the direction of $\nu_{\tilde{F}}$ is needed. If  for every point $x \in M$ there is a basis of tangent vectors for $T_{(x,0)} (B\times\mathbb{R})$ and graphical solitons with these as normal vectors then the harmonicity of the submersion $\pi$ is equivalent to the condition that, for every $u$, $F$ is a graphical soliton if and only if $\widetilde{F}$ is so.   
\end{remark}
From Theorem \ref{bundle2} we get immediately the following 
\begin{corollary}\label{cor_harm_sub}
Let $\pi:(M,g_M)\rightarrow (B,g_B)$ be a harmonic semi-Riemannian submersion. Let $u \in C^2(B)$, $F:B\rightarrow B\times\R$, $F(x)=(x,u(x))$ be its graph map, and $\widetilde{F}:M\rightarrow M\times\R$ be the map $\widetilde{F}(x)=(x,\pi\circ u(x))$. 
Then $\widetilde{F}$ is a graphical translating soliton if and only if $F$ is a graphical  translating soliton.
\end{corollary} 

In view of the preceding theorem, we will call $F$ an $H$-perturbed soliton. Obviously, by equation \eqref{fundamentalequation}, 
$F$ is an $H$-perturbed soliton if and only if $u$ satisfies
\begin{eqnarray}\label{h_perturbed_eq}
\mathrm{div}\left(\frac{ \grad u } { \sqrt{\ep'\big(\ep+\vert \grad u\vert_g^2\big)}} \right)=
\frac{1}{ \sqrt{\ep'\big(\ep+\vert \grad u\vert_g^2\big)}} -H_{\mathrm{fib}}^{\nu_{\widetilde{F}}}
\end{eqnarray}
We are now going to specialize to the case where the base $B$ is one-dimensional. We will show that in that case equation \eqref{h_perturbed_eq} can be reduced to a particular ODE.

We consider a semi-Riemannian submersion $\pi:(M,g_M)\rightarrow(I,\widetilde{\varepsilon}ds^2)$, where $I$ is an open interval. Needless to say, for each $s\in I$, the fiber $\pi^{-1}(s)$ is a hypersurface in $M$. We assume that each fiber has constant mean curvature (CMC), so that for each $s\in I$, we can call $h(s)$ the value of the mean curvature of $\pi^{-1}(s)$. 
Along the paper, we will  say that $h$ \textit{represents the mean curvature of the fibers.} 
\begin{theorem}\label{equiv_ODE}
Let $I$ be an open interval and $\pi:(M,g_M)\rightarrow(I,\widetilde{\varepsilon}ds^2)$ be a semi-Riemannian submersion with CMC fibers, and function $h$ representing the mean curvature of the fibers. Given  $u\in C^2(M)$ which is constant along the fibers of $\pi$, let  $\widetilde{F}:M\rightarrow M\times\R$, $\widetilde{F}(x)=(x,u(x))$, be its graph map, and $f:I\rightarrow \mathbb{R}$ be the map such that $u=f\circ\pi$. Then $\widetilde{F}$ is a translating soliton if and only if function $f$ is a solution to \eqref{keyODE}. 
\end{theorem}

\begin{proof}
By Theorem \ref{bundle2}, and since the base is one-dimensional, $F$ is a translating soliton if and only if 
\begin{eqnarray}\label{h_perturbed_eq_onedim}
\widetilde{\varepsilon}\left(\frac{f' } { \sqrt{\ep'\big(\ep+\widetilde{\varepsilon} f'^2\big)}} \right)'=
\frac{1}{ \sqrt{\ep'\big(\ep+\widetilde{\varepsilon}f'^2\big)}} -H_{\mathrm{fib}}^{\nu_{\widetilde{F}}}
\end{eqnarray}
Since $\pi$ is a semi-Riemannian submersion, $\vert\nabla\pi\vert^2=\widetilde{\varepsilon}$, where $\widetilde{\varepsilon}=\pm 1$ matches the signature of the base $I$. Then we have $(f\circ \pi)'=(f'\circ\pi)\nabla\pi$. After computing out the term on the left-hand side of equation \eqref{h_perturbed_eq_onedim}, we get therefore 
\begin{eqnarray*}
\widetilde{\varepsilon}\frac{f''} { \sqrt{\ep'\big(\ep+\widetilde{\varepsilon}f'^2\big)}}\left(1-\frac{\widetilde{\varepsilon}f'^2}{\ep+\widetilde{\varepsilon}f'^2}\right)=
\frac{1}{ \sqrt{\ep'\big(\ep+\widetilde{\varepsilon} f'^2\big)}} -H_{\mathrm{fib}}^{\nu_{\widetilde{F}}}
\end{eqnarray*}
Let us now point out the following fact about mean curvature. Let $H^N$ be the mean curvature of a submanifold $(Z,g_Z)$ with respect to a unit normal vector $N=\sum_ia_i\xi_i$, where $\{\xi_i\}$ are unit normal vectors to the submanifold  and $a_i$ are functions on the submanifold. Then if $\{e_i\}$ is an orthonormal basis of the tangent bundle to the submanifold we have by definition 
\[H^{N}=\sum_ig_Z(\nabla_{e_i}(\sum_ja_j\xi_j),e_i)=\sum_i\sum_ja_jg_Z(\nabla_{e_i}\xi_i,e_i)+g_Z(\partial_{e_i}(a)\xi_i,e_i)=\sum_ia_iH^{\xi_i}\]
We now compute the mean curvature $H_{\mathrm{fib}}^{\nu_{\widetilde{F}}}$ of the leaves $\pi^{-1}(x)$ in the direction of the normal vector in $M\times\mathbb{R}$. We have that
\[\nu_{\widetilde{F}}=\frac{1}{\sqrt{\varepsilon'\big((\nabla(f\circ\pi))^2+\varepsilon}\big)}(\nabla(f\circ\pi),-1)=\frac{1}{\sqrt{\varepsilon'\big((f'\circ\pi)^2|\nabla\pi|^2+\varepsilon\big)}}((f'\circ\pi)\nabla\pi-\partial_t)\]
Using the preceding computations, and by the fact that the mean curvature vanishes in the direction of $\partial_t$, we get consequently
\[H_{\mathrm{fib}}^{\nu_{\widetilde{F}}}=\frac{(f'\circ\pi)}{\sqrt{(f'\circ\pi)^2\widetilde{\varepsilon}+\varepsilon}}H_{\mathrm{fib}}^{\nabla\pi}.\]
Now the mean curvature of the fibers inside of $M$ is the mean curvature with respect to the unit normal vector $\nabla\pi$. Since the leaves have constant mean curvature, there exists a function $h:I\rightarrow \mathbb{R}$, such that $H_{\mathrm{fib}}^{\nabla\pi}=h\circ\pi$.
Hence equation \eqref{h_perturbed_eq_onedim} becomes after computing out the term on its left-hand side
\begin{eqnarray*}
\widetilde{\varepsilon}\frac{f''} { \sqrt{\ep'\big(\ep+\widetilde{\varepsilon}f'^2\big)}}\left(1-\frac{\widetilde{\varepsilon}f'^2}{\ep+\widetilde{\varepsilon}f'^2}\right)=
\frac{1}{ \sqrt{\ep'\big(\ep+\widetilde{\varepsilon} f'^2\big)}}-\frac{f'}{\sqrt{\ep'\big(\ep+\widetilde{\varepsilon} f'^2\big)}}h.
\end{eqnarray*}
and after simplifying
\begin{eqnarray*}
f''=(\te+\ep f'^2)(1-f'h), 
\end{eqnarray*}
which proves the result.
\end{proof}
\section{Lie Groups}

A particular example where the condition of the previous theorem are satisfied is the case of a manifold with a Lie group acting by isometries whose orbits give a foliation by codimension one submanifolds. Note that we have CMC fibers, so we can construct the map $h$ representing the mean curvature of the fibers. In this section $I$ will always be an open interval.

\begin{proposition}\label{prop_Liegroup_is_submersionwithCMCfibers}
Let $(M,g)$ be a connected semi-Riemannian manifold. Let $\Sigma$ be a Lie group acting by isometries on $M$ and $\pi:M\rightarrow I$ be a submersion such that the fibers of $\pi$ are orbits of the action. Moreover assume that the gradient of the projection $\nabla\pi$ is nowhere lightlike. Then, there exists $\widetilde{\ep} \in \{\pm 1\}$ and a map $v:I\rightarrow(\R,\widetilde{\ep} ds^2)$, such that $v\circ\pi$ is a semi-Riemannian submersion with constant mean curvature fibers.
\end{proposition}
\begin{proof} Since $\grad\pi$ is never zero or lightlike, $\mathrm{sign}(g(\grad\pi,\grad\pi))=\te=\pm 1$ is constant. 
In addition, since the fibers of $\pi$ are orbits of $\Sigma$, the length of $\grad\pi$ is constant along them. Then, let $v:I\rightarrow(\mathbb{R},\te ds^2)$ be a function such that $v'=\big(\sqrt{\widetilde{\ep}|\nabla\pi|^2}\big)^{-1}$. It follows that $v\circ\pi$ satisfies $|\nabla (v \circ \pi)|^2 = \widetilde{\ep}$.  Therefore it is a semi-Riemannian submersion.  Next, since the fibers of $\pi$ are orbits of the Lie group action, the mean curvature must be constant along them. As $v'$ has constant sign, $v$ must be injective, hence the fibers of $v \circ \pi$ are the same as the fibers of $\pi$.  It follows that its fibers have constant mean curvature  $H_{\mathrm{fib}}=\mathrm{div}(\nabla (v\circ\pi))$.
\end{proof}
From a technical point of view, the proposition shows that there is no loss of generality if we assume that $\vert\grad\pi\vert^2_g=\te$. and $\mathrm{div}(\grad \pi) = h\circ\pi=H^{\nabla{\pi}}_{\mathrm{fib}}$, for a suitable function $h:I\rightarrow\R$. Note that it is necessary for the Lie group to act by isometries. Indeed, consider the standard flat Riemannian metric on $\R^2$ and the action of the Boost Group on $\R^2$. Then, a simple computation shows that we cannot obtain the constancy of the length of $\grad\pi$ (see Example \ref{Magda}) .

We recall that a function $u:M\rightarrow\R$ is called \textit{invariant} by the Lie group $\Sigma$, or also $\Sigma$-invariant,  if it satisfies
\begin{equation}
\label{invariance} u:M\rightarrow\R, \quad u(x)=u(\sigma\cdot x), \ \forall x\in M, \ \forall \sigma\in \Sigma. 
\end{equation}
Accordingly, we will say that a graphical translating soliton is \textit{invariant} by the Lie group $\Sigma$, or also $\Sigma$-invariant, when its graph map is invariant by $\Sigma$.  We are now ready to prove the following

\begin{theorem} \label{existence}   Let $(M,g)$ be a connected semi-Riemannian manifold. Let $\Sigma$ be a Lie group acting by isometries on $M$ and $\pi:(M,g_M)\rightarrow (I,\widetilde{\varepsilon}ds^2)$ be a semi-Riemannian submersion such that the fibers of $\pi$ are orbits of the action, with function $h$ representing the mean curvature of the orbits. Take $u\in C^2(M,\R)$ and consider its graph map $F:M\rightarrow M\times\R$, $F(x)=(x,u(x))$ for any $x\in M$. Then, $F$ is a $\Sigma$-invariant translating soliton if, and only if, there exists a solution $f\in C^2(I,\R)$ to the ODE \eqref{keyODE} such that $u=f\circ \pi$.
\end{theorem}

\begin{proof}
The fact that $u$ is invariant by $\Sigma$ readily shows the existence of a smooth function $f: I\rightarrow \R$ such that $u=f\circ \pi=f(\pi)$. By using Proposition \ref{prop_Liegroup_is_submersionwithCMCfibers}, we can assume that $\pi:M\to I$ is a semi-Riemannian submersion with constant mean curvature fibers. Hence, Theorem \ref{existence} follows directly from Theorem \ref{equiv_ODE}.
\end{proof}
Proposition \ref{prop_Liegroup_is_submersionwithCMCfibers} actually shows that we can refine 
Theorem  \ref{existence} in the following more general, but also more technical way.

\begin{corollary} 
\label{existence2} Let $(M,g)$ be a connected semi-Riemannian manifold. Let $\Sigma$ be a Lie group acting by isometries on $M$ and $\pi:(M,g_M)\rightarrow (I,\widetilde{\varepsilon}ds^2)$ be a  submersion such that the fibers of $\pi$ are orbits of the action, with function $h:I\to\R$  representing the mean curvature of the fibers. Assume that the gradient of the projection $\nabla\pi$ is nowhere lightlike. Consider $u\in C^2(M,\R)$ and its graph map  $F:M\rightarrow M\times\R$, $F(x)=(x,u(x))$ for any $x\in M$. Then, $F$ is a $\Sigma$-invariant translating soliton if and only if there exists a solution $f\in C^2(I,\R)$ to the ODE \eqref{keyODE}, such that $u=f\circ v\circ\pi$, with $v:I\rightarrow\R$ a map satisfying $v'=\big(\sqrt{\widetilde{\ep}|\nabla\pi|^2}\big)^{-1}$.
\end{corollary}
Before studying examples, we want to give some results about the behavior of solutions of the ODE \eqref{keyODE} depending on the function $h$ and the signs $\ep$ and $\te$.
\begin{corollary} \label{oide} Let $(M,g)$ be a connected semi-Riemannian manifold. Let $\Sigma$ be a Lie group acting by isometries on $M$ and $\pi:M\rightarrow I$ be a submersion such that the fibers of $\pi$ are orbits of the action. Moreover assume that $\vert\grad \pi\vert^2=\te=\pm 1$ and $\mathrm{div}(\grad\pi)=h\circ\pi$, for some function $h: I\to\mathbb{R}$. Then for each $y_o\in \R$ and each $s_o\in I$ such that $h(s_o)\neq 0$, there exist a real number $\rho>0$ and a translating soliton $F:(y_0-\rho,y_o+\rho)\times\Sigma \rightarrow\bar{M}$ such that it is the union of two graphical translating solitons. 
\end{corollary}
\begin{proof} If $f'(s_o)\neq 0$, then, there is a small interval $J$ around $y_o=f(s_o)\in \R$ such that $\alpha=f^{-1}$ is well-defined on it. Given $y\in J$, we have according to \eqref{keyODE}, 
\[ -\frac{\alpha''(y)}{\alpha'(y)^3} = \left( \ep +\frac{\te}{\alpha'(y)^2}\right)\left(1-\frac{h(\alpha(y))}{\alpha'(y)}\right),
\]
which is 
\[ \alpha''(y) = \left(\te +\ep \alpha'(y)^2\right) \left( h(\alpha(y)) - \alpha'(y)\right).
\]
From here, we can follow the proof of Lemma 2.3 in \cite{CSS}. 
\end{proof}
\begin{corollary} \label{globalexamples}
Under the same conditions, assume that $\ep\te =-1$. Then, given $s_o\in I$, $f_1\in (-1,1)$ and $f_o\in \R$, there exists a solution $f:I\rightarrow\R$ to \eqref{keyODE} such that $f(s_o)=f_0$ and $f'(s_o)=f_1$.
\end{corollary}
\begin{proof} We make the change $w=f'$, so that \eqref{keyODE} reduces to 
\[w'(s) = \pm (1-w(s)^2) (1 - w(s)h(s)).
\]
Note that constant functions $w(s)=\pm 1$ are solutions to this differential equation, and they do not cross. Then, given initial conditions $s_o\in I$ and $f_1\in (-1,1)$, there exists a well-defined solution $w:I\rightarrow\R$. It remains to compute $f(s)=f_o+\int_{s_o}^s w(x)dx$ for some $f_o\in\R$.
\end{proof}

\section{Examples}

\begin{examp}  \label{cero} A generalization of the \textit{Grim Reaper Cylinder} in $\mathbb{R}^{m+1}$. \normalfont Given $(M,g_M)$ a connected semi-Riemannian manifold, let $\Omega$ be an open domain in $M$ such that $\Gamma:\Omega\rightarrow M\times\R$ is a translating soliton in $(M\times\R,g_M+\ep dt^2)$, for some $\ep=\pm 1$. Now, we consider another connected semi-Riemannian manifold $(P,g_P)$. The product 
\[ \bar{\Gamma}: P\times \Omega\rightarrow (P\times M\times\R,g_P+g_M+\ep dt^2) \quad
\bar{\Gamma}(q,x)=(q,\Gamma(x)),
\]
is just another translating soliton. Indeed, if $\vec{H}$ is the mean curvature vector of $\Gamma$, then, $\vec{J}=(0,\vec{H})$ is the mean curvature vector of $\bar{\Gamma}$, since we are dealing with the product spaces. Thus, $\vec{J}=\partial_t^{\perp}.$ 
$\Box$
\end{examp}

\begin{examp} \label{un} \normalfont We study those translating solitons in $\R^{n+1}$ with standard flat metric which are $SO(n)$-invariant, also known as \textit{rotationally invariant}. If in $\R^n$ with usual coordinates $x=(x_1,\ldots,x_n)$, we take $\pi(x)=\sqrt{\sum_i x_i^2}$, then 
\[ \grad \pi (x)= \sum_i \frac{x_i}{\pi(x)}\partial_i\vert_x, \quad \vert \grad\pi(x)\vert^2 = \sum_i\frac{x_i^2}{\pi(x)^2} = 1, \quad \mathrm{div}(\grad \pi)(x) = \frac{n-1}{\pi(x)}.
\]
This means that our manifold has to be $M=\R^n\backslash\{0\}$.  However, in order to obtain a $C^2$-class translating soliton,  \eqref{keyODE} becomes 
\begin{equation} \label{yasesabe}
f''(s)=\big(1+ f'(s)^2\big)\Big(1-\frac{n-1}{s}f'(s)\Big), \quad 
f'(0)=0, \ f(0)=a\in\R.
\end{equation}
In \cite{AW}, it was proved that there exist a convex, rotationally symmetric translating soliton over the plane. 
This was improved in \cite{CSS} obtaining entire  rotationally symmetric translating solitons $F:\mathbb{R}^{n}\times[0,+\infty)\rightarrow\mathbb{R}^{n+1}$, $n\geq 2$. In fact, in both papers \cite{AW} and \cite{CSS}, the same ODE as \eqref{yasesabe} is obtained. This means that this problem has a unique $C^{\infty}[0,+\infty)$ solution. The associated translating soliton is known as \textit{translating paraboloid}. In addition, by Corollary \ref{oide}, we recover the \textit{translating catenoid}.

In the book \cite{Besse}, we can find a list of Lie groups acting transitively and effectively on the sphere $\mathbb{S}^{n-1}$, hence, also on $\R^{n}$, and the space of leaves will be in all cases the interval $[0,+\infty)$. Thus, we can change $O(n)$ by any of them, namely $SO(n)$, $SU(n)$, $Sp(n)Sp(1)$, $Sp(n)U(1)$, $Sp(n)$, $G_2$, $Spin(7)$ and $Spin(9)$, for suitable values of the dimension $n$, depending on the case. This means that we can use all these groups to recover the rotationally symmetric translating solitons in $\R^n$. 
$\Box$
\end{examp}

In the following examples, we will need some tools, which can be found in the book \cite{Wiggins}. In $\R^2$, with coordinates $(s,t)\in\R^2$, 
an \textit{autonomous vector field} is a map $X:\R^2\rightarrow\R^2$, at least of class $C^1$. Given a point $p\in\R^2$ such that $X(p)=(0,0)$, we compute its \textit{linearlization} at $p$, namely 
\[ DX(p) = \begin{pmatrix} \frac{\partial X}{\partial s}(p) \  \frac{\partial X}{\partial x}(p) \end{pmatrix}.\]
We can compute the (complex) eigenvalues of this matrix, $\lambda_1$ and $\lambda_2$, with eigenvectors $v_1$ and $v_2$. When both eigenvalues are real and $\lambda_1\lambda_2<0$, the point $p$ is called a \textit{saddle point}. 
In such case, according to Theorem 3.2.1 of \cite{Wiggins}, there is a submanifold (in this case, just a curve) whose tangent space at $p$ is spanned by the eigenvector of negative eigenvalue, called the \textit{unstable submanifold}. The good property for us is that this submanifold can be seen as the graph of map defined on a small neighbourhood of $p$. 

\begin{examp}\label{deux} \normalfont Consider the Minkowski space $\mathbb{L}^{n+1}$ with standard flat metric $g=\sum_{i=1}^ndx_i^2-dx_{n+1}^2$. A rotationally invariant space-like soliton $F:\R^n\rightarrow\mathbb{L}^{n+1}$, $F(x)=(x,u(x))$ is determined by the group $O(n)$, in a very similar way as in Example \ref{un}, but  \eqref{keyODE}  is now
\begin{equation}\label{rotacionalL3}
f''(s)=\big(1 - f'(s)^2\big)\Big(1-\frac{n-1}{s}f'(s)\Big), \quad 
f'(0)=0, \ f(0)=a\in\R.
\end{equation}
To study the solution to this problem, first step is to define $w=f'$, obtaining
\begin{equation} \label{wsingular}
w'(s) = (1-w(s)^2)\Big(1-\frac{n-1}{s}w(s)\Big), \quad w(0)=0.
\end{equation}
We consider the following \textit{autonomous vector field}:
\[ X:\R^2\rightarrow\R^2, \quad 
X(s,x) = \Big(s, (n-1)x^3-s x^2-(n-1)x+s  \Big).
\]
Note that $X(0,0)=(0,0)$. At $(0,0)$, the \textit{linearlization} is
\[ DX(0,0) = \begin{pmatrix} \frac{\partial X}{\partial s}(0,0) \  \frac{\partial X}{\partial x}(0,0) \end{pmatrix}
= \begin{pmatrix} 1& 0 \\ 1 & 1-n  \end{pmatrix},
\]
whose eigenvalues are $\lambda_1=1$ and $\lambda_2=1-n$, with corresponding eigenvectors $v_1=(n,1)^t$, $v_2=(0,1)^t$. The point $(0,0)$ is a saddle point. By Theorem 3.2.1 of \cite{Wiggins}, there exist a $1$-dimensional,  
local unstable analytic manifold (\textit{of fixed points}), around $(0,0)$, 
whose tangent space at $(0,0)$ is spanned by $v_2$, which is a graph in an small interval around $s=0$, namely 
\[
W = \{ (s,x)\in \R\times\R : x=w(s), \ \vert s\vert\, small\},
\]
for some analytic function $w$ defined on a small interval $(-\bar{\delta},\bar{\delta})$.  
This means that our dynamical system has a solution $\alpha:(-\delta,\delta)\rightarrow W$, $\alpha(t)=(s(t),x(t))$, with $\alpha'(t)=X(\alpha(t))$,  which is an analytic diffeomorphish such that $\alpha(0)=(0,0)$, $\alpha'(0)=\lambda (n,1)$ for some $\lambda\in\R$, $\lambda\neq 0$, and $x(t)=w(s(t))$. We compose with the inverse of $s$, so that $t=t(s)$. Thus, $w(s)=x(t(s))$ is analytic near $0$, with $w(0)=0$. In addition,  since $X(\alpha(t))=\alpha'(t)$, then for $s>0$, 
\begin{gather*}
w'(s) = x'(t(s))t'(s) = \frac{x'(t(s))}{s'(t)} = \frac{(1-x(t(s))^2)(s(t)-(n-1)x(t(s))}{s(t)} \\
= (1-x(t(s))^2) \Big ( 1-\frac{(n-1)x(t(s))}{s(t)}\Big) = (1-w(s))^2) \Big ( 1-\frac{(n-1)w(s)}{s}\Big). 
\end{gather*}
From here, recalling $f'=w$, we obtain an analytic solution to \eqref{rotacionalL3}. Then, as in Corollary \ref{globalexamples}, we know that we can extend this solution to $f:[0,\infty)\rightarrow\R$, with $f'(0)=0$. 

In a similar fashion to Example \ref{un}, we can construct a spacelike rotationally symmetric translating soliton with two ends, that we can also call the \textit{translating catenoid} in $\mathbb{L}^{n+1}$. $\Box$
\end{examp}

\begin{corollary}\label{notimelikeinLn}
There do not exist time-like $SO(n)$-invariant translating solitons in Minkowski space $\mathbb{L}^{n+1}$.
\end{corollary}
\begin{proof} Let $F:\Omega\subset \R^n\rightarrow \mathbb{L}^{n+1}$, $F(x)=(x,u(x))$ a timelike $SO(n)$-invariant translating soliton. The action of $SO(n)$ in $\mathbb{L}^{n+1}$ is determined by the time-like axis $L=\{(0,x_{n+1}) \in\mathbb{L}^{n+1} : x_{n+1}\in\R\}$. For $F$ be $SO(n)$-invariant and smooth, then the tangent space at the intersection of $F(\Omega)\cap L$ has to be spacelike, which is a contradiction. 
\end{proof}

\begin{examp}\label{desitter} \normalfont 
We consider the Minkowski space $\mathbb{L}^{n+1}$, $n\geq 3$, with its usual flat metric $g(X,Y)=\sum_{i=1}^n X_iY_i-X_{n+1}Y_{n+1}$. The de Sitter space-time is
\[ dS^n = \{ x=(x_1,\ldots,x_{n+1}) \in \mathbb{L}^{n+1} : g(x,x)=+1 \}.
\]
As usual, we identify the tangent space at $x\in dS^n$, 
\[
T_x dS^n = \{ X=(X_1,\ldots,X_{n+1})\in\R^{n+1} : g(X,x)=0\}. 
\]
In particular, the position vector is a unit, spacelike, normal vector field $\chi:dS^n\rightarrow\mathbb{L}^{n+1}$. 
Now, the Lie group $O(n-1)$ acts by isometries on $dS^n$ as usual:
\[ O(n-1)\times dS^n \rightarrow dS^n, \quad 
(A,x)\mapsto A\cdot x = \left(\begin{array}{c|c} A & 0 \\  \hline 0 & 1 \end{array}\right)x = \left( \begin{array}{c} A(x_1,\ldots,x_n)^t \\ x_{n+1}\end{array}\right).
\]
We consider the map 
\[ \tau : dS^n\rightarrow \R, \quad \tau(x)=x_{n+1}.
\]
We define also de constant vector
\[ \xi=(0,\ldots,0,-1).\]
Clearly, 
\[ \tau = g(\xi,\chi).\]
Given a point $x\in dS^n$ and $X\in T_xdS^n$, we have $g_x(\grad\tau,X)=(d\tau)_xX = g(\xi,X)$. Therefore, $\grad\tau$ is going to be the tangent component of $\xi$, i.~e., 
\[ \grad\tau = \xi - g(\xi,\chi)\chi = \xi -\tau \chi, \quad  g(\grad\tau,\grad\tau) = -1-\tau^2.\]
We look for  smooth functions $v$ and $\bar{\tau}$ such that $\tau=v\circ\bar\tau$ and $\|\grad \bar{\tau}\|^2= -1$. Then, $\grad\tau=v'(\bar \tau)\grad\bar\tau$. This implies $(v'(\bar\tau))^2(-1)=-1-\tau^2=-1-(v(\bar{\tau}))^2$, so that $v'(s)=\pm\sqrt{1+v(s)^2}$. We choose $v(s)=\sinh(s)$. That means 
\[ \bar{\tau}  = \sinh^{-1}(\tau) \ \Leftrightarrow \tau=\sinh(\bar{\tau}).
\]

Consider a local orthonormal frame $(e_1,\ldots,e_n)$ of $TdS^n$, with $\ep_i=g(e_i,e_i)=\pm 1$. Let $\nabla$, $\bar{\nabla}$ be respectively the Levi-Civita connection of $dS^n$ and $\mathbb{L}^{n+1}$. 
\begin{align*}
\mathrm{div}(\grad\tau) &= \sum_i \ep_i g(\nabla_{e_i}\grad\tau,e_i) =  \sum_i \ep_i g(\bar\nabla_{e_i}\grad\tau,e_i) =
 \sum_i \ep_i g(\bar\nabla_{e_i}(\xi-\tau\chi),e_i) \\
 &= \sum_i \ep_i g(\bar\nabla_{e_i}(-\tau\chi),e_i) = \sum_i (-\ep_i)\tau  g(\bar\nabla_{e_i}(\chi),e_i) = 
 -\tau \sum_{i} \ep_i^2= -n\tau. 
\end{align*}
Since $\tau=v\circ\bar{\tau}$, then 
\begin{align*}
-n\sinh(\bar{\tau})=-n\tau = \mathrm{div}(\grad\tau) = v'(\bar{\tau}) \mathrm{div}(\grad\bar{\tau}) +v''(\bar\tau)\|\grad\bar{\tau}\|^2 = \cosh(\bar{\tau})\mathrm{div}(\grad\bar{\tau})-\sinh(\bar{\tau}).
\end{align*} 
Thus, 
\[\mathrm{div}(\bar{\tau}) = -(n-1)\tanh(\bar{\tau}).\]
By taking $h(s)=-(n-1)\tanh(s)$, the initial value problem becomes
\begin{equation}\label{pvi-desitter}
f''(s)=\big(-1+\ep f'(s)^2\big)\left(1+(n-1)\tanh(s)f'(s) \right), \  f'(s_o)=f_0, \ f(s_o)=f_1,
\end{equation}
for any $s_o\in\R$ and any initial values $f_o,f_1\in\R$.  

We point out that by Corollary \ref{globalexamples}, when $\ep=+1$, there exist solutions defined on the whole $\R$ whenever $f'(s_o)\in (-1,1)$. They give rise to entire rotationally symmetric translating solitons in $dS^{n}\times \R$. Since $\te=-1$,  this type of translating solitons will have $\ep' = \mathrm{sign}(\ep +\te (f')^2) =+1$, that it to say, all of them will be spacelike. In addition, by Corollary \ref{oide}, it is possible to construct spacelike rotationally symmetric translating solitons with two ends whose topology is $\mathbb{S}^{n-1}\times \R$.  
\end{examp}

\begin{examp}\label{hyperbolicspace} \normalfont 
We consider the Minkowski space $\mathbb{L}^{n+1}$, $n\geq 2$, with its usual flat metric $g(X,Y)=\sum_{i=1}^n X_iY_i-X_{n+1}Y_{n+1}$. The Weierstra\ss' model of the Hyperbolic Space is
\[ \W = \{ x=(x_1,\ldots,x_{n+1}) \in \mathbb{L}^{n+1} : g(x,x)=-1, \ x_{n+1}\geq 1\}.
\]
As usual, we identify the tangent space at $x\in \W$, 
\[
T_x \W = \{ X=(X_1,\ldots,X_{n+1})\in\R^{n+1} : g(X,x)=0\}. 
\]
In particular, the position vector is a unit, timelike, normal vector field $\chi:\W\rightarrow\mathbb{L}^{n+1}$. 
The Lie group $O(n-1)$ acts by isometries on $\W$ as usual:
\[ O(n-1)\times \W \rightarrow \W, \quad 
(A,x)\mapsto A\cdot x = \left(\begin{array}{c|c} A & 0 \\  \hline 0 & 1 \end{array}\right)x = \left( \begin{array}{c} A(x_1,\ldots,x_n)^t \\ x_{n+1}\end{array}\right).
\]
We consider the map 
\[ \tau : \W\rightarrow \R, \quad \tau(x)=x_{n+1}.
\]
We define also the constant vector
\[ \xi=(0,\ldots,0,-1).\]
Clearly, 
\[ \tau = g(\xi,\chi).\]
Given a point $x\in \W$ and $X\in T_x\W$, we have $g_x(\grad\tau,X)=(d\tau)_x X = g(\xi,X)$. Therefore, $\grad\tau$ is going to be the tangent component of $\xi$, i.~e., 
\[ \grad\tau = \xi + g(\xi,\chi)\chi = \xi+\tau \chi, \quad  g(\grad\tau,\grad\tau) =\tau^2 -1\geq 0.\]
The only critical point is $p=(0,\ldots,0,1)$, so we remove it and restrict all computations to $\Omega=\W\backslash\{p\}$. In addition, $\te = \mathrm{sign}(\grad\tau)=+1$.

We look for  smooth functions $v$ and $\bar{\tau}$ such that $\tau=v\circ\bar\tau$ and $\|\grad \bar{\tau}\|^2= 1$. Then, $\grad\tau = v'(\bar{\tau})\grad\bar{\tau}$. This implies $\tau^2-1=v(\bar{\tau})^2-1=(v'(\bar\tau))^2$, so that $v'(s)=\pm\sqrt{s^2-1}$. We choose $v(s)=\cosh(s)$. That means 
\[ \bar{\tau}  = \cosh^{-1}(\tau) \ \Leftrightarrow \tau=\cosh(\bar{\tau}).\]
Consider a local orthonormal frame $(e_1,\ldots,e_n)$ of $T\Omega$. Let $\nabla$, $\bar{\nabla}$ be respectively the Levi-Civita connection of $\Omega$ and $\mathbb{L}^{n+1}$. 
\begin{align*}
\mathrm{div}(\grad\tau) &= \sum_i  g(\nabla_{e_i}\grad\tau,e_i) 
=  \sum_i  g(\bar\nabla_{e_i}\grad\tau,e_i) = \sum_i g(\bar\nabla_{e_i}(\xi+\tau\chi),e_i) \\
 &= \sum_i  g(\bar\nabla_{e_i}(\tau\chi),e_i) = \sum_i \tau  g(\bar\nabla_{e_i}(\chi),e_i) = 
 \tau \sum_{i} 1= n\tau. 
\end{align*}
Since $\tau=v\circ\bar{\tau}$, then 
\begin{align*}
n\cosh(\bar{\tau})=n\tau = \mathrm{div}(\grad\tau) = v'(\bar{\tau}) \mathrm{div}(\grad\bar{\tau}) +v''(\bar\tau)\|\grad\bar{\tau}\|^2 
= \sinh(\bar{\tau})\mathrm{div}(\grad\bar{\tau})+\cosh(\bar{\tau}).
\end{align*} 
Thus, 
\[\mathrm{div}(\bar{\tau}) = (n-1)\coth(\bar{\tau}).\]
By taking $h(s)=(n-1)\coth(s)$, the initial value problem becomes
\begin{equation}\label{pvi-hyperbolicspace}
f''(s)=\big(1+\ep f'(s)^2\big)\left(1-(n-1)\coth(s)f'(s) \right), \  f'(0)=0, \ f(0)=f_1,
\end{equation}
for any  initial value $f_1\in\R$.
By taking $w=f'$, we consider the auxiliary problem 
\[ w'(s) = \big(1+\ep w(s)^2\big)\big(1-(n-1)\coth(s)w(s) \big), \  w(0)=0.
\]
We want to show that this problem has solutions. Thus, we proceed as in Example \ref{deux}. We consider the dynamical system
\[ X:\R^2\rightarrow\R^2, \quad X(s,x)=\Big(\sinh(s), (1+\ep x^2)\big(\sinh(s)-(n-1)\cosh(s)x\big)\Big).
\]
Note that $X(0,0)=(0,0)$. The \textit{linearization} at $(0,0)$ is
\[ DX(0,0) = \begin{pmatrix} \frac{\partial X}{\partial s}(0,0) \  \frac{\partial X}{\partial x}(0,0) \end{pmatrix}
= \begin{pmatrix} 1& 0 \\ 1 & 1-n  \end{pmatrix}.
\]
By repeating the steps in Example \ref{deux}, we obtain an analytic solution $f:[0,T)\rightarrow\R$ to \eqref{pvi-hyperbolicspace} for each $f_1\in\R$, $f(0)=f_1$.

When $\ep=-1$, by Corollary \ref{globalexamples}, $T=+\infty$. This means that we obtain entire rotationally symmetric graphical translating solitons  in $\W\times\R$ with a Lorentzian metric. In addtion, it is possible to construct a \textit{translation helicoid} in this setting.

Assume now that $\ep=+1$. Firstly, we see that $f'(s)\geq 0$ for any $s\in[0,T)$. Indeed, if there exists $s_o\in (0,T)$ such that $f'(s_o)<0$, by \eqref{pvi-hyperbolicspace}, we obtain $f''(s_o)>0$, which shows that $f'$ is increasing around $s_o$. Since $f'(0)=0 > f'(s_o)$, by continuity, there exists $s_1\in (0,s_o)$ such that $f''(s_1)=0$ and $f'(s_1)\leq f'(s_0)$. But then, by \eqref{pvi-hyperbolicspace}, we obtain $0=\big(1+f'(s_1)^2\big)\big( 1-(n-1)\coth(s_1)f'(s_1)\big)>0$, which is a contradiction. Next, we are going to see that we can extend $f$ from $0$ to $+\infty$.  For the sake of simpleness, we call $w=f':[0,T)\rightarrow\R$, and $g, G:[0,T]\times [0,\infty)\rightarrow\R$ given by $g(s,x)=(1+x^2)(1-(n-1)\coth(s)x)$ and $G(s,x)=(1+x^2)(1-(n-1)x/s)$. A simple computation shows 
\[G(s,x)-g(s,x)= (n-1)(1+x^2)\left(\coth(s)-\frac1s\right) x \geq 0, \quad \forall (s,x)\in [0,\infty)\times [0,\infty). 
\]
Thus, the solution to the problem $w'(s)=G(s,w(s))$, $w(0)=0$ is an upper bound of the solution to the problem $w'(s)=g(s,w(s))$, $w(0)=0$. But we already know that the solution to the problem 
\[ w(0)=0, \quad w'(s)= (1+w(s)^2)\left(1 - \frac{n-1}{s}w(s)\right)
\]
is defined on $[0,\infty)$. This readily shows that $f$, solution to \eqref{pvi-hyperbolicspace}, is defined on $[0,\infty)$.  $\Box$

\end{examp}

\begin{examp}\normalfont \label{principalfiberbundle} We recall 
\[ H_1^3 = \{ (z_1,z_2)\in \mathbb{C}^2 :  -\vert z_1\vert^2+\vert z_2\vert^2=-1 \}
\]
is a 3-dimensional quadric in the standard indefinite complex space $\mathbb{C}^2_1$. Also, it is well-known that the unit group of complex numbers $\mathbb{S}^1=\{ a \in \mathbb{C} : a\bar{a}=1 \}$ acts by isometries on $H_1^3$, $(a,(z_1,z_2))\to (az_1,az_2)$, obtaining a principal fiber bundle over the complex hyperbolic line $H_1^3\rightarrow \mathbb{C}H^1$ in the usual way, with timelike, totally geodesic fibers $\mathbb{S}^1$. In addition, $\mathbb{C}H^1$ is isometric to the real hyperbolic plane $\R H^2$. In fact, by denoting $\bar{z}$ the complex conjugate of $z\in\mathbb{C}$, 
we can construct the map
\[ \pi: H_1^3 \rightarrow  \mathbb{C}\times \R\equiv \R^3, \quad \pi(z_1,z_2) = \big(2z_1\bar{z}_2, \vert z_1\vert^2+\vert z_2\vert^2\big).
\]
The image of this map is the Weierstrass model of $\R H^2$, embedded in the Minkowski 3-space as a quadric, and therefore it is easy to regard $\pi$ as the projection from $H_1^3$ to $\mathbb{C}H^1\equiv\R H^2$. Next, in Example \ref{hyperbolicspace} we obtained rotationally symmetric translating solitons in $\R H^2$, with group $\Sigma=\mathbb{S}^1$. Now, by Corollary \ref{cor_harm_sub}, we obtain a translating soliton $\Gamma$ in $H_1^3$, by lifting the translating soliton from $\R H^2$. Note that $\Gamma$ is invariant  by the Lie group $\mathbb{S}^1\times\mathbb{S}^1$. $\Box$
\end{examp}

\begin{examp}\normalfont  \label{Magda}  
Take the Minkowski plane $\L^2$ with standard flat metric $g=dx^2-dy^2$. The \textit{boost} Lie group is the set of matrices
\[ \Sigma = \left\{ A_{\theta} = \begin{pmatrix} \cosh(\theta) & \sinh(\theta) \\
\sinh(\theta) & \cosh(\theta) \end{pmatrix} : \theta\in\R 
\right\},
\]
which acts on $\L^2$ by isometries. However, in order to obtain suitable quotients, we need to split the plane in four regions whose boundaries are made of two light-like geodesics, namely
\begin{align*}
&\Omega_1 =  \{ (x,y)\in\L^2 : y^2<x^2, \ 0<x \}, \quad 
&\Omega_2 =  \{ (x,y)\in\L^2 : y^2>x^2, \ 0<y \} \\
&\Omega_3 =  \{ (x,y)\in\L^2 : y^2<x^2, \ 0>x \}, \quad 
&\Omega_4 =  \{ (x,y)\in\L^2 : y^2>x^2, \ 0>y \}.
\end{align*}
We will use the globally defined orthonormal frame $\{\partial_x,\partial_y\}$. 
For $F$ be $\Sigma$-invariant, we have to extend the action of $\Sigma$ to $(\R^3,\produ{,}=g+dt^2)=\L^3$ in the natural way, namely
\[ \Sigma\times \R^3\rightarrow\R^3, \quad (A,(x,y,t))\mapsto \big((x,y)A,t\big).\]
Then, at the intersection point with the axis $\{(0,0)\}\times\R$, the tangent plane has to be orthogonal to the axis. In other words, $(dF)_{(0,0)} T_{(0,0)}\R^2 \perp \partial_t$. This means that $f'(0)=0$. \smallskip

We begin with $\Omega_1$. We define the projection, with its usual properties:
\begin{gather*}\pi:\Omega_1\rightarrow (0,+\infty),\ 
\pi(x,y)=+\sqrt{x^2-y^2},\\
\grad\pi(x,y)=\frac{x}{\pi(x,y)}\partial_x+\frac{y}{\pi(x,y)}\partial_y, \
\vert \grad\pi \vert^2 =+1=\te, \  
\mathrm{div}(\grad \pi) = 
\frac{1}{\pi}.
\end{gather*}
It is very simple to check that $\pi\big((x,y)A_{\theta}\big) = \pi(x,y)$ for any $(x,y)\in\L^2$ and $\theta\in\R$. This means that $\pi$ can work as the expected projection map. 

The product manifold $(\Omega_1\times\R,\produ{,}=g+dt^2)$ is an open subset of flat 3-Minkowski space. Thus, we are constructing a time-like translating soliton as a graph over a timelike 2-plane. We transform our differential equation into the following problem.

We recall that we need $W^2=\ep'(\ep+\te f'(\pi)^2)\geq 0$. By now, we know $\ep=\te=+1$. Then, \eqref{keyODE} becomes
\begin{equation} \label{caseOmega1}
f''(s)=\big(1+f'(s)^2\big)\left(1-\frac{f'(s)}{s}\right), \ f(0)=a, \ f'(0)=0.
\end{equation}
We already know that this problem has a $C^{\infty}[0,+\infty)$ solution by Example \ref{un}. 

We can construct our translating soliton $F^1:\Omega_1\rightarrow \L^3$, $F^1(x,y)=\Big(x,y,f\big(\sqrt{x^2-y^2}\big)\Big)$. 
Note that we have another similar one 
\[F^3:\Omega_3\rightarrow\L^3, \quad F^3(x,y)=\Big(x,y,f\big(\sqrt{x^2-y^2}\big)\Big).\]
In addition, by Corollary \ref{oide}, there exist \textit{boost invariant, spacelike, translating solitons with two ends}. \smallskip 

Next, we work on $\Omega_2$. Now, the projection $\pi$ satisfies
\begin{gather*}\pi:\Omega_2\rightarrow (0,+\infty),\ \pi(x,y)=+\sqrt{y^2-x^2},\\
\grad\pi(x,y)=\frac{-x}{\pi(x,y)}\partial_x-\frac{y}{\pi(x,y)}\partial_y, \
\vert \grad\pi \vert^2= -1=\te, \ 
\mathrm{div}(\grad \pi) = \frac{-1}{\pi}.
\end{gather*}
The differential equation  \eqref{keyODE} becomes now:
\begin{equation}\label{caseOmega2}
f''(s) = \big(-1+f'(s)^2\big)\left( 1 +\frac{f'(s)}{s}\right), \ f(0)=a, \ f'(0)=0.
\end{equation}
By the easy change $q(s)=-f(s)$, we transform this problem in
\[ q''(s) = \big(1-q'(s)^2\big)\left( 1 -\frac{q'(s)}{s}\right), \ q(0)=-a, \ q'(0)=0.
\]
By Example \ref{deux}, we know that this problem has an analytic solution in $[0,+\infty)$. 

For $a=0$, we call $f_1$ the solution to \eqref{caseOmega1} in $[0,+\infty)]$, and $f_2$ the solution to \eqref{caseOmega2} in $[0,+\infty)]$, so that we can define a continous function 
\[ u:\R^2\rightarrow \R, \quad u(x,y) = \left\{
\begin{array}{cl} f_1\left(\sqrt{x^2-y^2}\right),  & (x,y)\in \Omega_1\cup\Omega_3, \\
0,  & (x,y)\in\partial\Omega_i, i=1,2,3,4,\\
f_2\left(\sqrt{y^2-x^2}\right),  & (x,y)\in\Omega_2\cup\Omega_4.
\end{array}
\right.
\]
Note that we get immediately $u\in C^0(\R^2)\cap C^{\infty}(\cup_{i=1}^4\Omega_i)$.  We want to prove that $u$ is in $C^{\infty}$. Let $g(s):=f_1(is)$ in a small neighborhood of $0$,  where $i=\sqrt{-1}$. Then we have
$$[g(s)]'' = - f_1''(is) = - (1 + f_1'(is)^2)(1-\frac{f'_1(is)}{is})  = - (1 - g'(s)^2 )(1 + \frac{g'(s)}{s})$$
Hence $g$ is a solution to equation \eqref{caseOmega2}, and consequently $g=f_2$. But then we get comparing the derivatives of $f_1$ and $f_2$, that 
$f_1^{(i)}(0)=f_2^{(i)}(0)=0$, if $i$ is odd, since $f_1$ and $f_2$ are real, and $f_1^{(4i+2)}(0)=-f_2^{(4i+2)}(0)$ and $f_1^{(4i)}(0)=f_2^{(4i)}(0)$ for all $i\geq 0$. 

Notice that in the same way if $\widetilde{f_1}(-s):=f_1(s)$ and $\widetilde{f_2}(-s):=f_2(s)$ in a small neighborhood of $0$, similar computations show that $\widetilde{f_1}$ and  $\widetilde{f_2}$ are also a solution to \eqref{caseOmega1} and \eqref{caseOmega2} respectively. Therefore, $\widetilde{f_1}=f_1$ and $\widetilde{f_2}=f_2$ and $f_1$ and $f_2$ are even.  

We now just need to prove the following
\begin{lemma}
Let $f_1$, $f_2$ be functions in $C^{2n}(\R)$, such that $f_1^{(k)}(0)=f_2^{(k)}(0)=0$, if $k$ is odd, $f_1^{(k)}(0)=(-1)^{\frac{k}{2}}f_2^{(k)}(0)$, if $k$ is even. Then the function $u$ defined as above is in $C^{n}(\R^2)$.
\end{lemma}
\begin{proof}
We prove the statement by induction over $n$. The case $n=0$ is trivially satisfied, since $f_1(0)=f_2(0)$. Moreover let $g_1(z)=\frac{f'_1(z)}{z}$, $g_2(z)=\frac{f'_2(z)}{z}$. We have
\begin{eqnarray*}
\partial_xf_1(\sqrt{x^2-y^2})&=&xg_1(\sqrt{x^2-y^2}),\quad \partial_xf_2(\sqrt{y^2-x^2})=-xg_2(\sqrt{y^2-x^2})\\
\partial_yf_1(\sqrt{x^2-y^2})&=&-yg_1(\sqrt{x^2-y^2}),\quad \partial_yf_2(\sqrt{y^2-x^2})=yg_2(\sqrt{y^2-x^2})
\end{eqnarray*} 
Obviously $g_1$ and $g_2$ are in $C^{2n-2}$ since $f_1'(0)=f'_2(0)=0$, and $g_1^{(4i+2)}(0)=-g_2^{(4i+2)}(0)$ and $g_1^{(4i)}(0)=g_2^{(4i)}(0)$ for all $i\geq 0$. 
Then by the induction hypothesis the function \[ v:\R^2\rightarrow \R, \quad v(x,y) = \left\{
\begin{array}{cl} g_1\left(\sqrt{x^2-y^2}\right),  & (x,y)\in \Omega_1\cup\Omega_3, \\
0,  & (x,y)\in\partial\Omega_i, i=1,2,3,4,\\
g_2\left(\sqrt{y^2-x^2}\right),  & (x,y)\in\Omega_2\cup\Omega_4.
\end{array}
\right.
\]
is in $C^{n-1}(\R^2)$. Hence $\partial_xu$ and $\partial_yu$ extend to $C^{n-1}(\R^2)$, and finally $u$ extends to $C^{n}(\R^2)$.
\end{proof}
This proves that  $u$ is in $C^{\infty}$. As a final remark, we point out that the curve joining each two adjacent pieces is a lightlike straight line. Thus,  it is possible to consider two or three contiguous pieces, and their gluing straight lines. We cannot consider the other straight half-lines, since the boundary would be lightlike. In other words, we can choose to glue either two, three or four adjacent pieces to obtain solitons.
\end{examp}


\section*{Acknowledgements}

M.~Ortega has been partially financed by the Spanish Ministry of Economy and Competitiveness and European Regional Development Fund (ERDF), project MTM2016-78807-C2-1-P.  This project started during a visit of the second author to Imperial College in February 2016: M.~A.~Lawn and M.~Ortega would like to thank the London Mathematical Society and Imperial College London, since this visit was partially supported by a Scheme 4 Grant of the LMS.

\end{document}